\documentclass[a4paper,11pt]{article}%amsart}
\usepackage[utf8x]{inputenc}
\usepackage{a4wide}
\usepackage{amssymb}
\usepackage{amsmath}
\usepackage{amsthm}
\usepackage{latexsym}
\usepackage[mathscr]{euscript}

\theoremstyle{definition}
\newtheorem{definition}{Definition}[section]

\theoremstyle{plain}
\newtheorem{Theorem}[definition]{Theorem}
\newtheorem{Proposition}[definition]{Proposition}
\newtheorem{Lemma}[definition]{Lemma}
\newtheorem{Corollary}[definition]{Corollary}

\theoremstyle{remark}
\newtheorem{remark}[definition]{Remark}

\newcommand{\R}{\mathbb R}  
\newcommand{\N}{\mathbb N}

\newcommand{\grad}{\mathrm{grad}}

\newcommand{\eps}{\varepsilon}
\newcommand{\vphi}{\varphi}
\usepackage[color]{changebar}
\cbcolor{blue}

\title{A regularisation approach to causality theory for $C^{1,1}$-Lorentzian metrics}
\author{Michael Kunzinger\footnote{University of Vienna, Faculty of Mathematics, 
michael.kunzinger@univie.ac.at, roland.steinbauer@univie.ac.at, milena.stojkovic@live.com}, 
Roland Steinbauer\footnotemark[\value{footnote}], Milena Stojkovi\'c\footnotemark[\value{footnote}],
James A.\ Vickers\footnote{University of Southampton, School of Mathematics, J.A.Vickers@maths.soton.ac.uk}}

\begin{document}

\date{}

%\date{Received: date /Accepted: date}

\maketitle

\begin{abstract}

We show that many standard results of Lorentzian causality theory remain valid if the
regularity of the metric is reduced to $C^{1,1}$. Our approach is based on regularisations of 
the metric adapted to the causal structure.

\vskip 1em


\end{abstract}

\section{Introduction}

Traditionally, general relativity as a geometric theory has been formulated
for smooth space-time metrics. However, over the decades the PDE point of
view has become more and more prevailing. After all, general relativity as a
physical theory is governed by field equations and
questions of regularity are essential in the context of solving the initial
value problem. Already the classical local existence theorem for the vacuum
Einstein equations (\cite{CG69}) deals with space-time metrics in
$H^s_{\mbox{\small loc}}$ with $s>5/2$ (which merely guarantees the metric on
the spatial slices to be $C^1$) and more recent studies have significantly
lowered the regularity (\cite{KR05,M06,KRS12}).

Also from the physical point of view non-smooth solutions are of vital interest.
For example, one would like to study systems where different regions of
space-time have different matter contents, e.g. inside and outside a star, or in
the case of shock waves. On matching these regions the matter variables become
discontinuous, which via the field equations forces the differentiability of
the metric to be below $C^2$. E.g.\  a metric of regularity $C^{1,1}$ 
(continuously differentiable with locally Lipschitz first order
derivatives, often also denoted by $C^{2-}$) corresponds to finite jumps of the matter variables.
In the standard approach (\cite{L55}) one deals with metrics
which are piecewise $C^3$ but globally are only $C^1$. Even more extreme situations 
are exemplified by impulsive waves (e.g.\
\cite[Ch.\ 20]{GP09}) where the metric is still $C^3$ off the impulse but
globally is merely $C^0$.

On the other hand, in the bulk of the literature in general relativity it seems
to be assumed (sometimes implicitly) that the differentiability of the
space-time metric is at least $C^2$, especially so in the standard references on
causality theory. More precisely, the presentations
in \cite{P72,ON83,BEE96,Kriele,MS} generally (seem to) assume smoothness,
while \cite{HE,Seno1,Seno2,Chrusciel_causality} assume $C^2$-differentiability.
This mismatch in regularity---the quest for low regularity from physics
and analysis versus the need for higher regularity to maintain standard
results from geometry---has of course been widely noted, see
e.g.\ \cite{HE,MS,Clarke,Seno1,Seno2,Chrusciel_causality,SW} for a review of
various approaches to causal structures and discussions of
regularity assumptions. The background of this ``annoying problem''
(\cite[\S2]{Seno2}) is that for $C^2$-metrics the existence of totally normal (convex)
neighbourhoods is guaranteed. Furthermore, as emphasised by Senovilla,
$C^2$-differentiability of the metric is one of the fundamental assumptions of
the singularity theorems (see \cite[\S6.1]{Seno1} for a discussion of regularity 
issues in this context). Finally, in
\cite{Chrusciel_causality} it has recently been explicitly demonstrated that
assuming the metric to be $C^2$ allows one to retain many of the 
standard causality properties of smooth metrics. 

However, if one attempts to lower the differentiability of the metric below
$C^2$ one encounters serious problems. It \emph{is} possible to develop some of
the elements of causality theory in low regularity: E.g.,
smooth time functions exist on domains of dependence even for continuous metrics (\cite{FS11,CG}) and the
space of causal curves is still compact in this case (\cite{SW}). 
On the other hand it is well-known that some essential building
blocks of the theory break down for general $C^1$-metrics. Explicit
counterexamples by Hartman and Wintner, \cite{Hartman,HW} (in
the Riemannian case) show that for
connections of H\"older regularity $C^{0,\alpha}$ with $0<\alpha<1$ convexity
properties in small neighbourhoods may fail to hold. For example, radial
geodesics may fail to be minimising between any two points they contain. 
Also recently a study of the causality of continuous metrics in \cite{CG} has
revealed a dramatic failure of fundamental results of smooth causality:
e.g., light cones no longer need to be topological hypersurfaces of codimension
one. In fact, for any $0<\alpha<1$ there are metrics of regularity $C^{0,\alpha}$, called
`bubbling metrics', whose light-cones have nonempty
interior, and for whom the push-up principle ceases to hold (there exist
causal curves that are not everywhere null but for which there
is no fixed-endpoint deformation into a timelike curve).

For these reasons there has for some time been considerable interest in
determining the minimal degree of regularity of the metric for which standard
results of Lorentzian causality remain valid. A reasonable candidate is the
regularity class $C^{1,1}$ since it marks the
threshold where one still has unique solvability of the geodesic
equation, and the above remarks show that lower regularity will
in general prevent reasonable convexity properties. 
However, the main ingredient for studying local causality, the
exponential map, is now only locally Lipschitz and while  
it was well-known (\cite{Whitehead}) that it is a local homeomorphism, only
recently in \cite{KSS} it was shown to be in fact bi-Lipschitz. More precisely,
using approximation techniques and employing new methods of Lorentzian
comparison geometry (\cite{CleF}) it was shown in \cite[Th.\ 2.1]{KSS} that the
exponential map retains maximal regularity in the following sense:

%%%%%%%%%%%%%%%%%%%%%%%%%%%%%%%%%%%%%%%%%%%%%%%%%%%%%%%%%%%%%%%%%%%%%%%%%%%%%%%%%%%%%%%%%%%%%%%%%%%%%%%%%
%%%%%%%%%%%%%%%%%%%%%%%%%%%%%%%%%%%%%%%%%%%%%%%%%%%%%%%%%%%%%%%%%%%%%%%%%%%%%%%%%%%%%%%%%%%%%%%%%%%%%%%%%
\begin{Theorem}\label{mainpseudo}
  Let $M$ be a smooth manifold with a $C^{1,1}$-pseudo-Riemannian
  metric $g$ and let $p\in M$.  Then there exist open neighbourhoods
  $\tilde U$ of $0\in T_{p}M$ and $U$ of $p$ in M such that
\begin{equation*}
 \exp_{p}:\tilde U\rightarrow\ U
\end{equation*}
is a bi-Lipschitz homeomorphism.
\end{Theorem}
%%%%%%%%%%%%%%%%%%%%%%%%%%%%%%%%%%%%%%%%%%%%%%%%%%%%%%%%%%%%%%%%%%%%%%%%%%%%%%%%%%%%%%%%%%%%%%%%%%%%%%%%%
%%%%%%%%%%%%%%%%%%%%%%%%%%%%%%%%%%%%%%%%%%%%%%%%%%%%%%%%%%%%%%%%%%%%%%%%%%%%%%%%%%%%%%%%%%%%%%%%%%%%%%%%%
\medskip\noindent
It then follows from Rademacher's theorem that both $\exp_p$ and
$\exp_p^{-1}$ are differentiable almost everywhere. If $\exp_p: \tilde
U \to U$ is a bi-Lipschitz homeomorphism and $\tilde U$ is star-shaped
around $0$ we call $U$ a normal neighbourhood of $p$. If $U$ is a
normal neighbourhood of each of its elements then it is called totally
normal. In the literature (e.g., \cite{ON83}), totally normal sets are
also called convex sets. Any totally normal set $U$ is geodesically
convex in the sense that for any two points in $U$ there is a unique
geodesic contained in $U$ that connects them. Totally normal sets play
an important role in local causality theory, see Section
\ref{mainsec} below. The following result, proved in \cite[Th.\ 4.1]{KSS}
ensures that locally there always exist such neighbourhoods:
%%%%%%%%%%%%%%%%%%%%%%%%%%%%%%%%%%%%%%%%%%%%%%%%%%%%%%%%%%%%%%%%%%%%%%%%%%%%%%%%%%%%%%%%%%%%%%%%%%%%%%%%%
%%%%%%%%%%%%%%%%%%%%%%%%%%%%%%%%%%%%%%%%%%%%%%%%%%%%%%%%%%%%%%%%%%%%%%%%%%%%%%%%%%%%%%%%%%%%%%%%%%%%%%%%%
\begin{Theorem} \label{totally} Let $M$ be a smooth manifold with a
  $C^{1,1}$-pseudo-Riemannian metric $g$. Then each point $p\in M$
  possesses a basis of totally normal neighbourhoods.
\end{Theorem}
%%%%%%%%%%%%%%%%%%%%%%%%%%%%%%%%%%%%%%%%%%%%%%%%%%%%%%%%%%%%%%%%%%%%%%%%%%%%%%%%%%%%%%%%%%%%%%%%%%%%%%%%%
%%%%%%%%%%%%%%%%%%%%%%%%%%%%%%%%%%%%%%%%%%%%%%%%%%%%%%%%%%%%%%%%%%%%%%%%%%%%%%%%%%%%%%%%%%%%%%%%%%%%%%%%%
\noindent
The aim of this paper is to develop the key elements of causality theory for
$C^{1,1}$-Lorentzian metrics based on the above results as well as on refined
regularisation techniques, extending the approach of \cite{CG}, 
thereby demonstrating that indeed $C^{1,1}$ is the
minimal degree of regularity where a substantial part of 
smooth causality theory remains valid.

\medskip\noindent While we were in the final stages of preparing the
present paper we learned that an alternative approach to causality
theory for $C^{1,1}$-Lorentzian metrics by E.\ Minguzzi had recently
appeared in \cite{M}. This paper also establishes the fact
$\exp_p$ is a bi-Lipschitz homeomorphism, and in addition shows that
$\exp$ is a bi-Lipschitz homeomorphism on a neighbourhood of the zero-section
in $TM$ and is strongly differentiable over this zero section
\cite[Th.\ 1.11]{M}. In this work, the required properties of the
exponential map are derived from a careful analysis of the
corresponding ODE problem based on Picard-Lindel\"of approximations,
as well as from an inverse function theorem for Lipschitz maps. In
\cite{M} the author also goes on to establish the Gauss Lemma and to develop
the essential elements of $C^{1,1}$-causality, thereby obtaining many
of the results that are also contained in the present work, some even in
greater generality. 

\medskip\noindent
Nevertheless, we believe that our approach is of interest, 
and that in fact the approach in \cite{M} and ours nicely complement
each other, for the following reasons: Our methods are a direct
continuation
of the regularisation approach of P.\ Chrusciel and J.\ Grant
(\cite{CG}) and are completely independent from those employed in \cite{M}.
The basic idea is to approximate a given metric of low regularity
(which may be as low as $C^0$) by two
nets of smooth metrics ${\check g}_\epsilon$ and ${\hat g}_\epsilon$
whose light cones sandwich those of $g$. We then continue the line of
argument of \cite{CG,KSS} to establish the key results of causality theory
for a $C^{1,1}$-metric (thereby answering a corresponding question
in \cite{CG} which mainly motivated this work, namely whether the results 
of \cite{Chrusciel_causality} remain true for $C^{1,1}$-metrics). 
The advantage of these methods is that they quite easily adapt to
regularity below $C^{1,1}$, which as far as we can see is the natural
lower bound for the applicability of those employed in \cite{M}.
As an example, we note that the push-up lemmas from \cite{CG},
cf.\ Prop.\ \ref{push-up1} and \ref{nonnull} below, in fact even hold for 
$C^{0,1}$-metrics (or, more generally, for causally plain $C^0$-metrics),
whereas the corresponding results in \cite[Sec.\ 1.4]{M} require the
metric to be $C^{1,1}$.

\medskip\noindent
Furthermore, although considerable work still
needs to be done, we believe that the regularisation approach adopted
here, together with methods from Lorentzian comparison geometry as used
in \cite{CleF} and \cite{KSS}, will allow us to address some of the
other results required (such as curvature estimates, variational properties
of curves, and existence of focal points) in order to establish singularity
theorems for $C^{1,1}$-metrics, where so far only limited results are
available (\cite{Seno1}). Indeed, we note that the relevance of the  
kind of approximation techniques advocated in \cite{CG,KSS} for such 
questions was already pointed out in \cite[Sec.\ 8.4]{HE}.

\medskip\noindent The plan of the paper is as follows. In section 2 we
introduce the regularisation techniques
and show how they may be applied to establish the Gauss Lemma
(Theorem \ref{Gauss}) for a $C^{1,1}$-pseudo-Riemannian metric. Section 3 deals
with the key elements of $C^{1,1}$-causality theory and in Theorem \ref{lcb}
we again use regularisation methods to show that the local causal
structure is given by the image of the null cone under the exponential
map. This is then used to show
that if a causal curve from $p$ ends at a point in $\partial J^+(p)$
then it is a null geodesic. We then go on to deduce the basic
elements of causality theory using standard methods. Finally in
section 4 we refer to the results of \cite{CG} to show that all the
major building blocks are in place to follow the $C^2$-proofs as given
in \cite{Chrusciel_causality} to establish those elements of causality
theory that do not rely on continuity of the curvature.

%%%%%%%%%%%%%%%%%%%%%%%%%%%%%%%%%%%%%%%%%%%%%%%%%%%%%%%%%%%%%%%%%%%%%%%%%%%%%%%%%%%%%%%%%%%%%%%%%%%%%%%%%
%%%%%%%%%%%%%%%%%%%%%%%%%%%%%%%%%%%%%%%%%%%%%%%%%%%%%%%%%%%%%%%%%%%%%%%%%%%%%%%%%%%%%%%%%%%%%%%%%%%%%%%%%
%%%%%%%%%%%%%%%%%%%%%%%%%%%%%%%%%%%%%%%%%%%%%%%%%%%%%%%%%%%%%%%%%%%%%%%%%%%%%%%%%%%%%%%%%%%%%%%%%%%%%%%%%
\section{Regularisation techniques}\label{mainsec}
%%%%%%%%%%%%%%%%%%%%%%%%%%%%%%%%%%%%%%%%%%%%%%%%%%%%%%%%%%%%%%%%%%%%%%%%%%%%%%%%%%%%%%%%%%%%%%%%%%%%%%%%%
%%%%%%%%%%%%%%%%%%%%%%%%%%%%%%%%%%%%%%%%%%%%%%%%%%%%%%%%%%%%%%%%%%%%%%%%%%%%%%%%%%%%%%%%%%%%%%%%%%%%%%%%%
%%%%%%%%%%%%%%%%%%%%%%%%%%%%%%%%%%%%%%%%%%%%%%%%%%%%%%%%%%%%%%%%%%%%%%%%%%%%%%%%%%%%%%%%%%%%%%%%%%%%%%%%%
Throughout this paper we assume $M$ to be a $C^\infty$-manifold and only lower
the regularity of the metric. This is no loss of generality since any
$C^k$-manifold with $k\ge 1$ possesses a unique $C^\infty$-structure that is
$C^k$-compatible with the given $C^k$-structure on $M$ (see \cite[Th.\
2.9]{Hirsch}).

As already mentioned in the introduction a fundamental tool in our approach
is approximating a given metric of regularity $C^{1,1}$ by a net
$g_\eps$ of $C^\infty$-metrics, in the following sense:
%%%%%%%%%%%%%%%%%%%%%%%%%%%%%%%%%%%%%%%%%%%%%%%%%%%%%%%%%%%%%%%%%%%%%%%%%%%%%%%%%%%%%%%%%%%%%%%%%%%%%%%%%
%%%%%%%%%%%%%%%%%%%%%%%%%%%%%%%%%%%%%%%%%%%%%%%%%%%%%%%%%%%%%%%%%%%%%%%%%%%%%%%%%%%%%%%%%%%%%%%%%%%%%%%%%
\begin{remark}\label{approxrem}
We cover $M$ by a countable and locally finite collection of
relatively compact chart neighbourhoods and
denote the corresponding charts by $(U_i,\psi_i)$ ($i\in \N$). Let $(\zeta_i)_i$ be a
subordinate partition of unity with $\mathrm{supp}(\zeta_i)\Subset U_i$ (i.e., 
$\mathrm{supp}(\zeta_i)$ is a compact subset of $U_i$) for all $i$
and choose a family of cut-off functions $(\chi_i)_i\in\mathscr{D}(U_i)$ with
$\chi_i\equiv 1$ on a
neighbourhood of $\mathrm{supp}(\zeta_i)$. Finally, let $\rho\in
\mathscr{D}(\R^{n})$ be a test function with unit integral and define the
standard mollifier $\rho_{\eps}(x):=\eps^{-n}\rho\left (\frac{x}{\eps}\right)$
($\eps>0$). Then denoting by $f_*$ (resp.\ $f^*$) push-forward
(resp.\ pullback) under a map $f$, the following formula defines a family 
$(g_\eps)_\eps$ of smooth sections of $T^0_2(M)$ 
\[
 g_\eps:=\sum\limits_i\chi_i\,
 g_\eps^i:=\sum\limits_i\chi_i\,\psi_i^*\Big(\big(\psi_{i\,*} (\zeta_i\,
g)\big)*\rho_\eps\Big)
\]
which satisfies
 \begin{itemize}
 \item[(i)] $g_\eps$ converges to $g$ in the $C^1$-topology as $\eps\to 0$, and
 \item[(ii)] the second derivatives of $g_\eps$ are bounded, uniformly in $\eps$, on compact sets.
 \end{itemize}
 On any compact subset of $M$, therefore, for $\eps$ sufficiently
 small the $g_\eps$ form a family of pseudo-Riemannian metrics of the
 same signature as $g$ whose Riemannian curvature tensors $R_\eps$ are
 bounded uniformly in $\eps$.  Indeed, properties (i) and
 (ii) were the only ones required to derive all results given in \cite{KSS}.

\medskip\noindent
Also observe that the above procedure can be applied even to distributional
sections of any vector bundle $E\to M$ (using the corresponding vector bundle
charts) and that the usual convergence properties of smoothings via convolution 
are preserved.
\end{remark}

\medskip\noindent
To distinguish exponential maps stemming from metrics $g_\eps$, etc.,
we will write $\exp_p^{g_\eps}$, etc.. For brevity we will drop this
superscript for the $C^{1,1}$-metric $g$ itself, though.
%%%%%%%%%%%%%%%%%%%%%%%%%%%%%%%%%%%%%%%%%%%%%%%%%%%%%%%%%%%%%%%%%%%%%%%%%%%%%%%%%%%%%%%%%%%%%%%%%%%%%%%%%
%%%%%%%%%%%%%%%%%%%%%%%%%%%%%%%%%%%%%%%%%%%%%%%%%%%%%%%%%%%%%%%%%%%%%%%%%%%%%%%%%%%%%%%%%%%%%%%%%%%%%%%%%
We shall need the following properties of the exponential maps
corresponding to an approximating net as above:
%%%%%%%%%%%%%%%%%%%%%%%%%%%%%%%%%%%%%%%%%%%%%%%%%%%%%%%%%%%%%%%%%%%%%%%%%%%%%%%%%%%%%%%%%%%%%%%%%%%%%%%%%
%%%%%%%%%%%%%%%%%%%%%%%%%%%%%%%%%%%%%%%%%%%%%%%%%%%%%%%%%%%%%%%%%%%%%%%%%%%%%%%%%%%%%%%%%%%%%%%%%%%%%%%%%
\begin{Lemma}\label{uniform} 
  Let $g$ be a $C^{1,1}$-pseudo-Riemannian metric on $M$ and let
  $g_\eps$ be a net of smooth pseudo-Riemannian metrics that satisfy
  conditions (i) and (ii) of Remark \ref{approxrem}. Then any $p\in M$ has a
  basis of normal neighbourhoods $U$ such that, with $\exp_p: \tilde U
  \to U$, all $\exp_p^{g_\eps}$ are diffeomorphisms with domain
  $\tilde U$ for $\eps$ sufficiently small. Moreover, the inverse maps
  $(\exp_p^{g_\eps})^{-1}$ also are defined on a common neighbourhood
  of $p$ for $\eps$ small, and converge locally uniformly to
  $\exp_p^{-1}$.
\end{Lemma}
%%%%%%%%%%%%%%%%%%%%%%%%%%%%%%%%%%%%%%%%%%%%%%%%%%%%%%%%%%%%%%%%%%%%%%%%%%%%%%%%%%%%%%%%%%%%%%%%%%%%%%%%%
%%%%%%%%%%%%%%%%%%%%%%%%%%%%%%%%%%%%%%%%%%%%%%%%%%%%%%%%%%%%%%%%%%%%%%%%%%%%%%%%%%%%%%%%%%%%%%%%%%%%%%%%%
\begin{proof} 
  The claims about the common domains of $\exp_p^{g_\eps}$, resp.\ of
  $(\exp_p^{g_\eps})^{-1}$ follow from \cite[Lemma 2.3 and
  2.8]{KSS}. To obtain the convergence result, we first note that
  without loss, given a common domain $V$ of the
  $(\exp_p^{g_\eps})^{-1}$ for $\eps<\eps_0$, we may assume that
  $\bigcup_{\eps<\eps_0} (\exp_p^{g_\eps})^{-1}(V)$ is relatively
  compact in $\tilde U$: this follows from the fact that the maps
  $(\exp_p^{g_\eps})^{-1}$ are Lipschitz, uniformly in $\eps$ (see
  \cite{KSS}, the argument following Lemma 2.10).

\medskip\noindent
  Now if $(\exp_p^{g_\eps})^{-1}$ did not converge uniformly to
  $\exp_p^{-1}$ on some compact subset of $V$ then by our compactness
  assumptions we could find a sequence $q_k$ in $V$ converging to some
  $q\in V$ and a sequence $\eps_k\searrow 0$ such that
  $w_k:=(\exp_p^{g_{\eps_k}})^{-1}(q_k)\to w \not =
  \exp_p^{-1}(q)$. But since $(\exp_p^{g_\eps}) \to \exp_p$ locally
  uniformly (by \cite[Lemma 2.3]{KSS}), we arrive at $q_k = \exp_p^{g_{\eps_k}}(w_k) \to
  \exp_p(w)\not=q$, a contradiction.
\end{proof}

\medskip\noindent
In the particular case of $g$ being Lorentzian, a more sophisticated
approximation procedure, adapted to the causal structure of $g$, was
given in \cite[Prop.\ 1.2]{CG}. 

\medskip\noindent
To formulate this result, we first recall that a space-time is a time-oriented
Lorentzian manifold (of signature $(-+\dots+)$), with time-orientation determined by some continuous timelike
vector field.  In what follows, all Lorentzian manifolds will be supposed to be
time-oriented. Also we recall from \cite{CG} that for two Lorentzian metrics $g$,
$h$, we say that $h$ has strictly larger light cones than $g$, denoted by $g\prec h$, if for any tangent
vector $X\not=0$, $g(X,X)\le 0$ implies that $h(X,X)<0$.

\medskip\noindent
We will also need the following technical tools:
\begin{Lemma}\label{sbe} Let $(K_m)$ be an exhaustive sequence of compact subsets of a manifold $M$
($K_m\subseteq K_{m+1}^\circ$, $M=\bigcup_m K_m$), and let $\eps_1\ge \eps_2 \ge \dots >0$ be given.
Then there exists some $\psi\in C^\infty(M)$ such that $0<\psi(p)\le \eps_m$ for $p\in K_m\setminus K_{m-1}^\circ$
(where $K_{-1}:=\emptyset$).
\end{Lemma}
\begin{proof} See, e.g., \cite[Lemma 2.7.3]{GKOS}.
\end{proof}

For what follows, recall that $K\Subset M$ denotes that $K$ is a compact subset of $M$. 
\begin{Lemma}\label{globallem} Let $M$, $N$ be manifolds, and set $I:=(0,\infty)$. 
Let $u: I\times M \to N$ be a smooth map
and let (P) be a property attributable to values $u(\eps,p)$, satisfying:
\begin{itemize}
\item[(i)] For any $K\Subset M$ there exists some $\eps_K>0$ such that (P) holds for 
all $p\in K$ and $\eps<\eps_K$.
\item[(ii)] (P) is stable with respect to decreasing $K$ and $\eps$:
if $u(\eps,p)$ satisfies (P) for all  $p\in K\Subset M$ and
all $\eps$ less than some $\eps_K>0$ then for any compact set $K'\subseteq K$
and any $\eps_{K'} \le \eps_K$, $u$ satisfies (P) on $K'$ for all
$\eps\le \eps_{K'}$.
\end{itemize}
Then there exists a smooth map $\tilde u: I\times M \to N$
such that (P) holds for all $\tilde u(\eps,p)$ ($\eps\in I$, $p\in M$) and
for each $K\Subset M$ there exists some $\eps_K\in I$ such that
$\tilde u(\eps,p) = u(\eps,p)$ for all $(\eps,p)\in (0,\eps_K] \times K$.
\end{Lemma}
\begin{proof} See \cite[Lemma 4.3]{HKS}.
\end{proof}
\medskip\noindent
Based on these auxiliary results, we can prove the following refined version
of \cite[Prop.\ 1.2]{CG}:
%%%%%%%%%%%%%%%%%%%%%%%%%%%%%%%%%%%%%%%%%%%%%%%%%%%%%%%%%%%%%%%%%%%%%%%%%%%%%%%%%%%%%%%%%%%%%%%%%%%%%%%%%
%%%%%%%%%%%%%%%%%%%%%%%%%%%%%%%%%%%%%%%%%%%%%%%%%%%%%%%%%%%%%%%%%%%%%%%%%%%%%%%%%%%%%%%%%%%%%%%%%%%%%%%%%
\begin{Proposition}\label{CGapprox} Let $(M,g)$ be a space-time with a
continuous Lorentzian metric, and $h$ some smooth
background Riemannian metric on $M$. Then for any $\eps>0$, there exist smooth
Lorentzian metrics $\check g_\eps$ and $\hat g_\eps$ on $M$ such that $\check g_\eps
\prec g \prec \hat g_\eps$ and $d_h(\check g_\eps,g) + d_h(\hat g_\eps,g)<\eps$,
where  
$$
d_h(g_1,g_2) := \sup_{0\not=X,Y\in TM} \frac{|g_1(X,Y)-g_2(X,Y)|}{\|X\|_h \|Y\|_h}.
$$
Moreover, $\hat g_\eps$ and $\check g_\eps$ depend smoothly on $\eps$, and if
$g\in C^{1,1}$ then $\check g_\eps$ and $\hat g_\eps$ additionally satisfy 
(i) and (ii) from Rem.\ \ref{approxrem}.
\end{Proposition}
%%%%%%%%%%%%%%%%%%%%%%%%%%%%%%%%%%%%%%%%%%%%%%%%%%%%%%%%%%%%%%%%%%%%%%%%%%%%%%%%%%%%%%%%%%%%%%%%%%%%%%%%%
%%%%%%%%%%%%%%%%%%%%%%%%%%%%%%%%%%%%%%%%%%%%%%%%%%%%%%%%%%%%%%%%%%%%%%%%%%%%%%%%%%%%%%%%%%%%%%%%%%%%%%%%%
%%%%%%%%%%%%%%%%%%%%%%%%%%%%%%%%%%%%%%%%%%%%%%%%%%%%%%%%%%%%%%%%%%%%%%%%%%%%%%%%%%%%%%%%%%%%%%%%%%%%%%%%%
%%%%%%%%%%%%%%%%%%%%%%%%%%%%%%%%%%%%%%%%%%%%%%%%%%%%%%%%%%%%%%%%%%%%%%%%%%%%%%%%%%%%%%%%%%%%%%%%%%%%%%%%%
\begin{proof} First we use time-orientation to obtain a $C^{1,1}$
timelike one-form $\tilde\omega$ (the $g$-metric equivalent of a smooth timelike vector field).
Using the smoothing procedure of Rem.\ \ref{approxrem}, on each $U_i$ we can pick $\eps_i>0$ so small
that $\tilde\omega_{\eps_i}$ is timelike on $U_i$. Then $\omega := \sum_i \zeta_i \tilde \omega_{\eps_i}$ is a
%which we smoothen according to the procedure detailed in Remark \ref{approxrem} to obtain a 
smooth timelike one-form on $M$. By compactness we obtain on every $U_i$ a constant $c_i>0$ such that 
\begin{equation}\label{ci}
   |\omega(X)|\geq c_i\quad \text{for all $g$-causal vector fields $X$ 
  with $\|X\|_h=1$}.
\end{equation}
\medskip\noindent
Next we set on each $U_i$ and for $\eta>0$ and $\lambda<0$ 
\begin{equation}\label{gel}
 {\hat g}^i_{\eta,\lambda}=g_\eta^i+\lambda\, \omega\otimes\omega,
\end{equation}
where $g^i_\eta$ is as in Remark \ref{approxrem} (set $\eps:=\eta$ there and $g^i_\eta:=g_\eta|_{U_i}$).
Let $\Lambda_k$ ($k\in\N$) be a compact exhaustion of $(-\infty,0)$. For each $k$, there exists
some $\eta_k>0$ such that $\eta_k< \min_{\lambda\in \Lambda_k} |\lambda|$, $\eta_k>\eta_{k+1}$ for all $k$, and

\begin{equation}\label{geps-g}
 |g^i_\eta(X,X)-g(X,X)|\leq |\lambda|\,\frac{c_i^2}{2}
\end{equation}
for all $g$-causal vector fields $X$ on $U_i$ with $\|X\|_h=1$, all $\lambda \in \Lambda_k$, and all $0< \eta\le \eta_k$. 
Thus by Lemma \ref{sbe} there exists a smooth function $\lambda\mapsto \eta(\lambda,i)$ on $(-\infty,0)$
with $0<\eta(\lambda,i)\le |\lambda|$ and such that \eqref{geps-g} holds
for all $g$-causal vector fields $X$ on $U_i$ with $\|X\|_h=1$, all $\lambda$, and all $0< \eta\le \eta(\lambda,i)$.

\medskip\noindent
Combining
\eqref{ci} with \eqref{geps-g} we obtain
\[
 {\hat g}^i_{\eta,\lambda}(X,X)
 =g(X,X)+(g^i_\eta-g)(X,X)+\lambda\,\omega(X)^2
 \leq0+\Big(|\lambda|\frac{c_i^2}{2}+\lambda c_i^2\Big)\|X\|_h^2<0,
\]  
for all $g$-causal $X$ and hence $g \prec {\hat g}^i_{\eta,\lambda}$  
for all $\lambda<0$ and $0<\eta\leq\eta(\lambda,i)$. 

\medskip\noindent
Given a compact exhaustion $E_k$ ($k\in \N$) of $(0,\infty)$, for each $k$ there exists
some $\lambda_k<0$ such that $|\lambda_k|<\min_{\eps\in E_k} \eps$,
$\lambda_k<\lambda_{k+1}$ for all $k$, and
\[
 d_{U_i}(\hat g^i_{\eta(\lambda,i),\lambda},g)
 :=\sup_{0\not=X,Y\in TU_i}\,\frac{|\hat g^i_{\eta(\lambda,i),\lambda}(X,Y)-g(X,Y)|}{
   \|X\|_h\|Y\|_h}
 <\,\frac{\eps}{2^{i+1}}.
\]
for all $\eps\in E_k$ and all $\lambda_k\le \lambda <0$. Again by Lemma \ref{sbe} we
obtain a smooth map $(0,\infty) \to (-\infty,0)$, $\eps\to \lambda_i(\eps)$ such that $|\lambda_i(\eps)|< \eps$
for all $\eps$, and
$d_{U_i}(\hat g^i_{\eta(\lambda_i(\eps),i),\lambda_i(\eps)},g)<\frac{\eps}{2^{i+1}}$
for all $\eps>0$.
We now consider the smooth symmetric $(0,2)$-tensor field on $M$,
\[
  g_\eps:=\sum_i\chi_i \hat g^i_{\eta(\lambda_i(\eps),i),\lambda_i(\eps)}.
\]
By construction, $(\eps,p)\mapsto g_\eps(p)$ is smooth, and $g_\eps$ converges to $g$
locally uniformly as $\eps\to 0$. Therefore, for any $K\Subset M$
there exists some $\eps_K$ such that for all $0<\eps<\eps_K$, $g_\eps$ is of the same
signature as $g$, hence a Lorentzian metric on $K$,
with strictly larger lightcones than $g$. We are thus in a position to apply Lemma \ref{globallem}
to obtain a smooth map $(\eps,p)\mapsto \hat g_\eps(p)$ such that for each fixed $\eps$, $\hat g_\eps$ is a globally
defined Lorentzian metric which on any given $K\Subset M$ coincides with $g_\eps$ for
sufficiently small $\eps$.

\medskip\noindent 
Then $d_h(\hat g_\eps,g)< \eps/2$, and $\eps\to 0$ implies $ \lambda_i(\eps) \to 0$ and a fortiori $\eta(\lambda_i(\eps),i)\to 0$
for each $i\in \N$. 

\medskip\noindent
From this, by virtue of \eqref{gel}, (i) and (ii) of Remark 
\ref{approxrem} hold for $\hat g_\eps$ if $g\in C^{1,1}$.

\medskip\noindent
The approximation $\check g_\eps$ is constructed analogously choosing $\lambda>0$.
\end{proof}

\begin{remark} 
\begin{itemize}
\item[(i)] From Rem.\ \ref{approxrem} and the above proof it follows that, given a
Lorentzian metric of some prescribed regularity (e.g., Sobolev, H\"older, etc.), 
the inner and outer regularisations $\check g_\eps$ and $\hat g_\eps$ converge to $g$ 
as good as regularisations by convolution do locally. 
\item[(ii)] If $g$ is a metric of general pseudo-Riemannian signature, then 
since $g_\eps$ in Rem.\ \ref{approxrem} depends smoothly on $\eps$, also in this case an application
of Lemma \ref{globallem} allows 
to produce regularisations $\tilde g_\eps$ that are pseudo-Riemannian metrics on all of $M$
of the same signature as $g$ and satisfy (i) and (ii) from that remark.
\end{itemize}
\end{remark}

To conclude this section we derive the Gauss lemma for $C^{1,1}$-metrics. 
This result has first appeared (in a more general form) in \cite{M}. We 
include an independent proof\footnote{The original proof which appeared in the 
  published 
  version as well as in arXiv:1310.4404v1--3[math.DG] is incorrect for it 
  wrongly 
  supposes that the map $s\mapsto \exp_p(t(v+sw))$ is a geodesic 
  (line 5 of the proof).
  %
  %\par
  Here we supply a new proof which is no longer based on   
  regulari\-sing the $\mathcal{C}^{1,1}$-metric but on a careful 
analysis of differentiability properties of the exponential map, which allows 
us to directly follow the classical proof.}.

\begin{Theorem}\label{Gauss} (The Gauss Lemma) 
  Let $g$ be a $C^{1,1}$-pseudo-Riemannian metric on $M$, and let
  $p\in M$.  Then $p$ possesses a basis of normal neighbourhoods $U$ with the following
properties: $\exp_p: \tilde U \to U$ is a bi-Lipschitz
  homeomorphism, where $\tilde U$ is an open star-shaped neighbourhood
  of $0$ in $T_pM$. Moreover,
  for almost all $x\in \tilde U$, if $v_x$, $w_x\in T_x(T_pM)$ and
  $v_x$ is radial, then
$$
\langle T_x\exp_p(v_x), T_x\exp_p(w_x)\rangle = \langle v_x, w_x \rangle.
$$
\end{Theorem} 
%%%%%%%%%%%%%%%%%%%%%%%%%%%%%%%%%%%%%%%%%%%%%%%%%%%%%%%%%%%%%%%%%%%%%%%%%%%%%%%%%%%%%%%%%%%%%%%%%%%%%%%%%
%%%%%%%%%%%%%%%%%%%%%%%%%%%%%%%%%%%%%%%%%%%%%%%%%%%%%%%%%%%%%%%%%%%%%%%%%%%%%%%%%%%%%%%%%%%%%%%%%%%%%%%%%
\begin{proof}
We may set $M=\R^n$ and assume that $x=v_x=v$ and that 
$T_v\exp_p$ exists. We then consider the mapping
\begin{equation}\label{1}
  f(t,s):=\exp_p\big(t(v+sw)\big),\qquad t\in[0,1]=:I,\ s\in[-\eps,\eps]=:J.
\end{equation}
We clearly have $ f_t(1,0)=T_v\exp_p(v)$ and $f_s(1,0)=T_v\exp_p(w)$
and it remains to show that $\langle f_t(1,0),f_s(1,0)\rangle=\langle 
v,w\rangle$.

To begin with note that $t\mapsto f(t,s)$ is a geodesic with initial speed 
$v+sw$. Hence $f_{tt}(t,s)=0$ and so we have for all $(t,s)\in I\times J$
\begin{equation}\label{2}
  \langle f_t(t,s),f_t(t,s)\rangle=\langle f_t(0,s),f_t(0,s)\rangle=\langle 
  v+sw,v+sw\rangle.
\end{equation}
On the other hand, by standard results from ODE theory, we have
\begin{equation}
  s\mapsto f(.,s)\ \in\ \mathrm{Lip}\big(J,\mathcal{C}^2(I,\R^n)\big).
\end{equation}
For any $1<p<\infty$, 
$\mathrm{Lip}\big(J,\mathcal{C}^2(I,\R^n)\big)\subseteq\mathrm{Lip}\big(J,W^{2,p
}((0,1),\R^n)\big)$.
Moreover, since the Sobolev space $W^{2,p}$ is reflexive for these values of 
$p$, it possesses 
the Radon-Nikodym property and so by \cite[p.\ 259]{LP} we have
\begin{equation}
  s\mapsto f_s(.,s)\in L^\infty\big(J,W^{2,p}((0,1),\R^n))\big)\subseteq
  L^\infty\big(J,\mathcal{C}^1_{\mathrm{b}}((0,1),\R^n)\big),
\end{equation}
where we have used the Sobolev embedding theorem. Thus 
for almost all $s$ and all $t$, $\partial_t\partial_s f(t,s)$ exists and is 
essentially bounded. Similarly $s\mapsto 
f_t(.,s)\in\mathrm{Lip}\big(J,\mathcal{C}^1(I,\R^n)\big)\subseteq 
\mathrm{Lip}\big(J,W^{1,p}((0,1),\R^n)\big)$ and again by the Radon-Nikodym 
property and the Sobolev embedding
\begin{equation}
  s\mapsto \partial_s\partial_t f(.,s) \in 
  L^\infty\big(J,\mathcal{C}^0_{{\mathrm b}}((0,1),\R^n)\big)\subseteq 
  L^\infty(I\times J).
\end{equation}
But now $\partial_t\partial_s f$  and  $\partial_s\partial_t f$ agree as 
distributions and hence also in ${L^\infty(I\times J)}$. Moreover the mapping 
$s\mapsto \langle f_s,f_t\rangle \in L^\infty(J,\mathcal{C}^1_{{\mathrm 
    b}}(I))$ 
and we 
may finally calculate
\begin{equation}\label{6}
  \partial_t \langle f_s,f_t\rangle = \langle f_{st},f_t\rangle
  = \langle f_{ts},f_t\rangle =\frac{1}{2}\,\partial_s\langle f_t,f_t\rangle 
  = \langle v,w\rangle +s \langle w,w\rangle,
\end{equation}
where we have used (\ref{2}). Observe that equation (\ref{6}) holds for all $t$ 
and hence also for all $s$ and so we have $\partial_t \langle 
f_s,f_t\rangle(t,0)=\langle v,w\rangle$ for all $t$. Furthermore since 
$f_s(0,0)$ and hence $\langle f_s,f_t\rangle(0,0)$ vanishes we find
$\langle f_s,f_t\rangle (t,0)=t\langle v,w\rangle$ and finally $\langle 
f_t(1,0),f_s(1,0)\rangle=\langle 
v,w\rangle$. 
\end{proof}
\section{Causality theory}

As in \cite{Chrusciel_causality} we will base our approach to
causality theory on locally Lipschitz curves. 
We note that this definition differs from that in \cite{M}, where 
the corresponding curves are required to be $C^1$ (see, however, Cor.\ \ref{lipisc1} below or \cite[Th.\ 1.27]{M}).
Any locally Lipschitz curve $c$ is
differentiable almost everywhere (Rademacher's theorem) and we call
$c$ timelike, causal, spacelike or null, if $c'(t)$ has the
corresponding property whenever it exists.  If the time-orientation of
$M$ is determined by a continuous timelike vector field $X$ then a causal curve
$c$ is called future- resp.\ past-directed if $\langle
X(c(t)),c'(t)\rangle < 0$ resp.\ $>0$ wherever $c'(t)$ exists. With
these notions we have:
\begin{definition}\label{causalitydef} 
Let $g$ be a $C^0$-Lorentzian metric on $M$. For $p\in A \subseteq M$ 
we define the relative chronological, respectively causal future of $p$ in $A$ by
(cf.\ \cite[2.4]{Chrusciel_causality}):
\begin{equation*}
\begin{split}
I^{+}(p,A) &:=\{q\in A |\ \text{there exists a future directed timelike curve in $A$ from $p$ to $q$  }\} \\
J^{+}(p,A) &:=\{q\in A |\ \text{there exists a future directed causal curve in $A$ from $p$ to $q$  }\}\cup A. 
\end{split}
\end{equation*}
For $B\subseteq A$ we set $I^+(B,A) := \bigcup_{p\in B} I^+(p,A)$ and analogously for $J^+(B,A)$.
We set $I^+(p):=I^+(p,M)$. Replacing `future directed' by `past-directed' we obtain the corresponding
definitions of the chronological respectively causal pasts $I^-$, $J^-$. 
\end{definition}
\medskip\noindent
Below we will formulate all results for $I^+$, $J^+$. By symmetry, the corresponding claims for
chronological or causal pasts follow in the same way.

\medskip\noindent
As usual, for $p$, $q\in M$ we write $p < q$, respectively $p\ll q$, if
there is a future directed causal, respectively timelike, curve from $p$ 
to $q$. By $p\le q$ we mean $p=q$ or $p<q$. 

\medskip\noindent
We now recall some definitions that were introduced in \cite{CG} and 
results there obtained which will be of use in this paper. 

\begin{definition}
A locally Lipschitz curve $\alpha: [0,1]\rightarrow M$ is said to be
locally uniformly timelike (l.u.-timelike) with respect to the $C^0$-metric $g$ 
if there exists a smooth 
Lorentzian metric $\check g\prec g$ such that $\check g(\alpha',\alpha')<0$ 
almost everywhere. Then for $p\in A\subseteq M$
\begin{equation*}
\check I^+_g(p,A):=\{q\in A|\ \text{there exists a future directed 
l.u.-timelike curve in $A$ from $p$ to $q$} \}. 
\end{equation*}
\end{definition}
Thus $\check I^+_g(A)=\bigcup_{\check g\prec g}I^+_{\check g}(A)$, hence it is open (\cite[Prop.\ 1.4]{CG}).
The following definition (\cite[Def.\ 1.8]{CG}) introduces a highly useful substitute for
normal coordinates in the context of metrics of low regularity
\begin{definition}\label{def:cylindrical chart}
Let $(M,g)$ be a smooth Lorentzian manifold with continuous metric $g$ 
and let $p\in M$. A relatively compact open subset $U$ of $M$ is called
a cylindrical neighbourhood of $p\in U$ if there exists a smooth chart
$(\vphi,U)$, $\vphi= (x^0,...,x^{n-1})$ with $\vphi(U)=I\times V$,
$I$ an interval around $0$ in $\R$ and $V$ open in $\R^{n-1}$, such that:
\begin{enumerate}
\item $\frac{\partial}{\partial x^0}$ is timelike and $\frac{\partial}{\partial x^i}$,
$i=1,...,n-1$, are spacelike,
\item For $q\in U,\ v\in T_qM$, if $g_q(v,v)=0$ then $\frac{|v^0|}{\|\vec{v}\|}\in (\frac{1}{2},2)$
(where $T_q\vphi(v)=(v^0,\vec{v})$, and $\|\,\|$ is the Euclidean norm on $\R^{n-1}$),
\item $(\varphi_\ast g)_{\vphi(p)}=\eta$ (the Minkowski metric).
\end{enumerate}
\end{definition}
By \cite[Prop.\ 1.10]{CG}, every point in a spacetime with continuous metric possesses a basis
of cylindrical neighbourhoods.
According to \cite[Def.\ 1.16]{CG}, a Lorentzian manifold $M$ with $C^0$-metric $g$ is called causally plain if for
every $p\in M$ there exists a cylindrical neighbourhood $U$ of $p$ such
that $\partial \check I^\pm (p,U)=\partial J^\pm (p,U)$. This condition excludes
causally `pathological' behaviour (bubbling metrics).
By \cite[Cor.\ 1.17]{CG}, we have:
\begin{Proposition}\label{lipplain}
Let $g$ be a $C^{0,1}$-Lorentzian metric on M. Then $(M,g)$ is 
causally plain. 
\end{Proposition}
The most important property of causally plain Lorentzian manifolds for our purposes is 
given in the following result (\cite[Prop.\ 1.21]{CG}).
\begin{Proposition} \label{ichecki}
Let $g$ be a continuous, causally plain Lorentzian metric and let $A\subseteq M$. Then 
\begin{equation}
I^\pm(A)=\check I^\pm(A). 
\end{equation}
\end{Proposition}
Furthermore, we will make use of the following `push-up' results (\cite[Lemma 1.22]{CG}, 
\cite[Prop.\ 1.23]{CG}):
\begin{Proposition}\label{push-up1} Let $g$ be a causally plain $C^0$-Lorentzian metric on $M$ and
let $p,\, q,\, r\in M$ with $p\le q$ and $q\ll r$ or $p\ll q$ and $q\le r$.  Then $p\ll r$.
\end{Proposition}
%%%%%%%%%%%%%%%%%%%%%%%%%%%%%%%%%%%%%%%%%%%%%%%%%%%%%%%%%%%%%%%%%%%%%%%%%%%%%%%%%%%%%%%%%%%%%%%%%%%%%%%%%
%%%%%%%%%%%%%%%%%%%%%%%%%%%%%%%%%%%%%%%%%%%%%%%%%%%%%%%%%%%%%%%%%%%%%%%%%%%%%%%%%%%%%%%%%%%%%%%%%%%%%%%%%
\begin{Proposition}\label{nonnull} 
Let $M$ be a spacetime with a
continuous causally plain metric $g$. Consider a causal future-directed curve $\alpha:[0,1]\to M$
from $p$ to $q$. If there exist $s_1$, $s_2\in [0, 1]$, $s_1 < s_2$, such that $\alpha|_{[s_1,s_2]}$ 
is timelike, then in any neighbourhood of $\alpha([0,1])$ there exists a timelike future-directed curve from $p$ to $q$.
\end{Proposition}

\medskip\noindent
Returning now to our main object of study, for the remainder of the paper 
$g$ will denote a $C^{1,1}$-Lorentzian metric. Then in particular, $g$ is 
causally plain by Prop.\ \ref{lipplain}.
To analyse the local causality for $g$ in terms of
the exponential map we first introduce some terminology. Let $\tilde
U$ be a star-shaped neighbourhood of $0\in T_pM$ such that $\exp_p:
\tilde U \to U$ is a bi-Lipschitz homeomorphism (Th.\
\ref{mainpseudo}). On $T_pM$ we define the position vector field
$\tilde P: v \mapsto v_v$ and the quadratic form $\tilde Q: T_pM \to
\R$, $v\mapsto g_p(v,v)$. By $P$, $Q$ we denote the push-forwards of
these maps via $\exp_p$, i.e.,
\begin{equation*}
\begin{split}
P(q) &:= T_{\exp_p^{-1}(q)}\exp_p (\tilde P (\exp_p^{-1}(q)) )\\
Q(q) &:= \tilde Q (\exp_p^{-1}(q)).
\end{split}
\end{equation*}
As $\exp_p$ is locally Lipschitz, $P$ is an $L^\infty_{\text{loc}}$-vector field on $U$, while $Q$ is
locally Lipschitz (see, however, Rem.\ \ref{regrem} below).

\medskip\noindent
Let $X$ be some smooth vector field on $U$ and denote by $\tilde X$ its pullback $\exp_p^*X$ (note that
$T_v\exp_p$ is invertible for almost every $v\in\tilde U$).
Then by Th.\ \ref{Gauss}, for almost every $q\in U$ we have, setting $\tilde q := \exp_p^{-1}(q)$:
\begin{align*}
\langle \grad Q(q), X(q) \rangle &=X(Q)(q)=\tilde{X}(\tilde Q)(\tilde q)=\langle \grad \tilde{Q},\tilde{X}\rangle|_{\tilde q}=
2\langle \tilde{P},\tilde{X} \rangle|_{\tilde q}=2\langle P,X \rangle|_{q}.
\end{align*}
It follows that $\grad Q = 2P$. 
\medskip\noindent

\begin{remark}\label{regrem}
It is proved in \cite{M} that the regularity of both $P$ and $Q$ is better than would
be expected from the above definitions. Indeed, \cite[Prop.\ 2.3]{M} even shows that $P$, as
a function of $(p,q)$ is strongly differentiable on a neighbourhood of the diagonal in
$M\times M$, and by \cite[Th.\ 1.18]{M}, $Q$ is in fact $C^{1,1}$ as a
function of $(p,q)$. We will however not
make use of these results in what follows and only remark that slightly weaker 
regularity properties of $P$ and $Q$ (as functions of $q$ only) can also be obtained
directly from standard ODE-theory. In fact, setting $\alpha_v(t):=\exp_p(tv)$ for $v\in T_pM$,
it follows that $P(q)=\alpha_{v_q}'(1)$, where $v_q:=\exp_p^{-1}(q)$.
Since $t\mapsto (\alpha_v(t),\alpha_v'(t))$ is the solution
of the first-order system corresponding to the geodesic equation with initial value $(p,v)$, 
and since the right-hand side of this system is Lipschitz-continuous, \cite[Th.\ 8.4]{Amann}
shows that $v\mapsto \alpha_v'(1)$ is Lipschitz-continuous. Since also $q\mapsto v_q$
is Lipschitz, we conclude that $P$ is Lipschitz-continuous.
From this, by the above calculation, it follows that $Q$ is $C^{1,1}$.
\end{remark}

\medskip\noindent
As in the smooth case, we may use $\exp_p$ to introduce normal
coordinates. To this end, let $e_0$, \dots, $e_n$ be an orthonormal
basis of $T_pM$ and for $q\in U$ set $x^i(q)e_i:=\exp_p^{-1}(q)$.  The
coordinates $x^i$ then are of the same regularity as $\exp_p^{-1}$,
i.e., locally Lipschitz.  The coordinate vector fields
$\left.\frac{\partial}{\partial x^i}\right|_q =
T_{\exp_p^{-1}(q)}\exp_p(e^i)$ themselves are in
$L^\infty_{\text{loc}}$. Note, however, that in the $C^{1,1}$-setting
we can no longer use the relation $g_p = \eta$ (the Minkowski-metric
in the $x^i$-coordinates), since it is not clear a priori that
$\exp_p$ is differentiable at $0$ with
$T_0\exp_p=\text{id}_{T_pM}$\footnote{See, however, \cite{M} where it
  is shown that indeed $\exp_p$ is even strongly differentiable at $0$
  with derivative $\text{id}_{T_pM}$.}. Due to the additional loss in
regularity it is also usually not advisable to write the metric in
terms of the exponential chart (the metric coefficients in these
coordinates would only be $L^\infty_{\text{loc}}$).

%%%%%%%%%%%%%%%%%%%%%%%%%%%%%%%%%%%%%%%%%%%%%%%%%%%%%%%%%%%%%%%%%%%%%%%%%%%%%%%%%%%%%%%%%%%%%%%%%%%%%%%%%
%%%%%%%%%%%%%%%%%%%%%%%%%%%%%%%%%%%%%%%%%%%%%%%%%%%%%%%%%%%%%%%%%%%%%%%%%%%%%%%%%%%%%%%%%%%%%%%%%%%%%%%%%
The following is the main result on the local causality in normal neighbourhoods.
\begin{Theorem}\label{lcb}
Let $g$ be a $C^{1,1}$-Lorentzian metric, and let $p\in M$. Then $p$ has a basis of normal neighbourhoods $U$, 
$\exp_p: \tilde U\to U$ a bi-Lipschitz homeomorphism, such that:
 \begin{equation*}
 \begin{split}
   I^{+}(p,U)=\exp_{p}(I^{+}(0)\cap \tilde{U})\\
   J^{+}(p,U)=\exp_{p}(J^{+}(0)\cap \tilde{U}) \\
   \partial I^{+}(p,U) = \partial J^{+}(p,U) =\exp_{p}(\partial I^{+}(0)\cap \tilde{U})
 \end{split}  
 \end{equation*}
Here, $I^+(0)=  \{v\in T_pM \mid \tilde Q(v)<0 \}$, and $J^+(0)= \{v\in T_pM \mid \tilde Q(v)\le  0 \}$.
In particular, $I^+(p,U)$ (respectively $J^+(p,U)$) is open (respectively closed) in $U$.
\end{Theorem}
%%%%%%%%%%%%%%%%%%%%%%%%%%%%%%%%%%%%%%%%%%%%%%%%%%%%%%%%%%%%%%%%%%%%%%%%%%%%%%%%%%%%%%%%%%%%%%%%%%%%%%%%%
%%%%%%%%%%%%%%%%%%%%%%%%%%%%%%%%%%%%%%%%%%%%%%%%%%%%%%%%%%%%%%%%%%%%%%%%%%%%%%%%%%%%%%%%%%%%%%%%%%%%%%%%%
\begin{proof} We first note that the third claim follows from the first two and the fact that
$\exp_p$ is a homeomorphism on $U$. For the proof of the first two claims we take a normal neighbourhood $U$ that is contained in a cylindrical neighbourhood of $p$.
In addition, we pick a regularising net $\hat g_\eps$ as
in Prop.\ \ref{CGapprox} and let $U$, $\tilde U$ as in Lemma
\ref{uniform} (fixing a suitable $\eps_0>0$).

\noindent  $(\supseteq)$
Let $v\in \tilde U$ and let $\alpha:= t\mapsto \exp_p(tv)$, $t\in [0,1]$.
Set $\alpha_\eps(t):=\exp_p^{\hat g_\eps}(tv)$. Then by continuous dependence on initial data 
we have that $\alpha_\eps\rightarrow \alpha$ in $C^{1}$ (cf.\ \cite[Lemma 2.3]{KSS}).
Hence applying the smooth Gauss lemma for each $\eps$ it follows that for each $t\in [0,1]$
we have
$$
g(\alpha'(t),\alpha'(t)) = \lim_{\eps\to 0}  \hat g_\eps(\alpha_\eps'(t),\alpha_\eps'(t))
%= \lim_{\eps\to 0} g_\eps(T_{tv}\exp^{ g_\eps}_{p}(v),T_{tv}\exp^{ g_\eps}_{p}(v)) =
= \lim_{\eps\to 0} (\hat g_\eps)_p(v,v) = g_p(v,v).
$$ 
Also, time-orientation is respected by $\exp_p$ since both $I(0)\cap \tilde U$ and 
$I(p,U)$ (by \cite[Prop.\ 1.10]{CG}) have two connected components,
and the positive $x^0$-axis in $\tilde U$ is mapped to $I^+(p,U)$.

\medskip\noindent
$(\subseteq)$:   We denote the
position vector fields and quadratic forms corresponding to 
$\hat g_\eps$ by $\tilde P_\eps$, $P_\eps$ and 
$\tilde Q_\eps$, $Q_\eps$, respectively.

\medskip\noindent
If $\alpha:[0,1]\to U$ is a future-directed causal curve in $U$
emanating from $p$ then $\alpha$ is timelike with respect to each
$\hat g_\eps$. Set $\beta:=(\exp_p)^{-1}\circ\alpha$ and $\beta_\eps:=
(\exp^{\hat g_\eps}_p)^{-1}\circ \alpha$. By \cite[Prop.\
2.4.5]{Chrusciel_causality}, $\beta_\eps([0,1]) \subseteq I^+_{\hat
g_\eps(p)}(0)$ for all $\eps<\eps_0$. Then by Lemma \ref{uniform}
we have that $\beta_\eps \to \beta$ uniformly, and that
$\tilde Q_\eps\to \tilde Q$ locally uniformly, so $\tilde Q(\beta(t)) = \lim
\tilde Q_\eps(\beta_\eps(t)) \le 0$ for all $t\in [0,1]$, and therefore
$\beta((0,1])\subseteq J^+(0)\cap \tilde U$. Together with the first 
part of the proof it follows that $\exp_p(J^+(0)\cap \tilde U) = J^+(p,U)$.
Now assume that $\alpha$ is timelike. Then by Prop.\ \ref{ichecki}, 
$\alpha((0,1])\subseteq \check{I}^{+}(p,U)$. 
This means that there exists a smooth metric $\check{g}\prec g$ such
that $\alpha$ is $\check{g}$-timelike. Let $f_{\check{g}}, f_{g}$ 
denote the graphing functions of $\partial I^{+}_{\check{g}}(p,U)$ 
and $\partial J^{+}(p,U)$, respectively (in a cylindrical chart, see \cite[Prop.\ 1.10]{CG}). 
Then by \cite[Prop.\ 1.10]{CG}, since $\alpha$ lies in $I^{+}_{\check{g}}(p,U)$,
it has to lie strictly above $f_{\check{g}}$, hence also strictly above $f_{g}$,
and so $\alpha((0,1])\cap \partial J^{+}(p,U)=\emptyset$. 
But then, since $\exp_p$ is a homeomorphism on $U$, we have that 
$$ \beta((0,1])\cap (\partial J^+(0)\cap \tilde U) =
\beta((0,1])\cap \exp_{p}^{-1}(\partial J^{+}(p,U))= 
 \exp_p^{-1}(\alpha((0,1])\cap \partial J^{+}(p,U))=\emptyset$$
Hence $\beta$ lies entirely in $I^+(0)\cap \tilde U$, as claimed.
\end{proof}
%%%%%%%%%%%%%%%%%%%%%%%%%%%%%%%%%%%%%%%%%%%%%%%%%%%%%%%%%%%%%%%%%%%%%%%%%%%%%%%%%%%%%%%%%%%%%%%%%%%%%%%%%
%%%%%%%%%%%%%%%%%%%%%%%%%%%%%%%%%%%%%%%%%%%%%%%%%%%%%%%%%%%%%%%%%%%%%%%%%%%%%%%%%%%%%%%%%%%%%%%%%%%%%%%%%
\begin{Corollary}\label{lipisc1} 
  Let $U\subseteq M$ be open, $p\in U$. Then the sets $I^+(p,U)$,
  $J^+(p,U)$ remain unchanged if, in Def.\ \ref{causalitydef},
  Lipschitz curves are replaced by piecewise $C^1$ curves, or in fact
  by broken geodesics.
\end{Corollary}
%%%%%%%%%%%%%%%%%%%%%%%%%%%%%%%%%%%%%%%%%%%%%%%%%%%%%%%%%%%%%%%%%%%%%%%%%%%%%%%%%%%%%%%%%%%%%%%%%%%%%%%%%
%%%%%%%%%%%%%%%%%%%%%%%%%%%%%%%%%%%%%%%%%%%%%%%%%%%%%%%%%%%%%%%%%%%%%%%%%%%%%%%%%%%%%%%%%%%%%%%%%%%%%%%%%
\begin{proof} 
  Let $\alpha: [0,1] \to U$ be a, say, future directed timelike
  Lipschitz curve in $U$.  By Th.\ \ref{totally} and Th.\ \ref{lcb} we
  may cover $\alpha([0,1])$ by finitely many totally normal open sets
  $U_i\subseteq U$, such that there exist $0=t_0 < \dots <t_N = 1$
  with $\alpha([t_i,t_{i+1}])\subseteq U_{i+1}$ and
  $I^{+}(\alpha(t_i),U_i)=\exp_{\alpha(t_i)}(I^{+}(0)\cap
  \tilde{U_i})$ for $0\le i < N$. Then the concatenation of the radial
  geodesics in $U_i$ connecting $\alpha(t_i)$ with $\alpha(t_{i+1})$
  gives a timelike broken geodesic from $\alpha(0)$ to $\alpha(1)$ in
  $U$.
\end{proof}
\medskip\noindent
The following analogue of \cite[Cor.\ 2.4.10]{Chrusciel_causality} provides more
information about causal curves intersecting the boundary of
$J^+(p,U)$:
%%%%%%%%%%%%%%%%%%%%%%%%%%%%%%%%%%%%%%%%%%%%%%%%%%%%%%%%%%%%%%%%%%%%%%%%%%%%%%%%%%%%%%%%%%%%%%%%%%%%%%%%%
%%%%%%%%%%%%%%%%%%%%%%%%%%%%%%%%%%%%%%%%%%%%%%%%%%%%%%%%%%%%%%%%%%%%%%%%%%%%%%%%%%%%%%%%%%%%%%%%%%%%%%%%%
\begin{Corollary}\label{boundary} 
  Under the assumptions of Th.\ \ref{lcb}, suppose that $\alpha: [0,1]
  \to U$ is causal and $\alpha(1)\in \partial J^+(p,U)$. Then $\alpha$
  lies entirely in $\partial J^+(p,U)$ and there exists a
  reparametrisation of $\alpha$ as a null-geodesic segment.
\end{Corollary}
%%%%%%%%%%%%%%%%%%%%%%%%%%%%%%%%%%%%%%%%%%%%%%%%%%%%%%%%%%%%%%%%%%%%%%%%%%%%%%%%%%%%%%%%%%%%%%%%%%%%%%%%%
%%%%%%%%%%%%%%%%%%%%%%%%%%%%%%%%%%%%%%%%%%%%%%%%%%%%%%%%%%%%%%%%%%%%%%%%%%%%%%%%%%%%%%%%%%%%%%%%%%%%%%%%%
\begin{proof} 
Suppose to the contrary that there exists $t_0\in (0,1)$ such that
$\alpha(t_0)\in I^+(p,U)$. Then there exists a future directed
timelike curve $\gamma$ from $p$ to $\alpha(t_0)$. 
Applying Prop. \ref{nonnull} to the concatenation
$\gamma\cup \alpha|_{[t_0,1]}$ it follows that there exists a future
directed timelike curve from $p$ to $\alpha(1)$. But then 
$\alpha(1)\in I^+(p,U)$, a contradiction. Thus $\alpha(t)\in \partial 
J^+(p,U),\ \forall t\in[0,1]$,
implying that $\beta(t)=\exp_p^{-1}\circ\alpha(t)\in \partial J^+(0),\ \forall t\in[0,1]$, so
$\tilde{Q}(\beta(t))=0,\ \forall t\in [0,1]$ and for almost all $t$ we have
\begin{equation*}
0=\frac{d}{dt}\tilde Q(\beta(t))=g_p(\grad \tilde Q(\beta(t)),\beta'(t))
=g_p(2\tilde P(\beta(t)),\beta'(t)).
\end{equation*}
Hence $\beta(t)$ is collinear with $\beta'(t)$
almost everywhere, and it is easily seen that this implies 
the existence of some $v\neq 0, v\in \partial J^+(0)$, and of some
$h:\R\rightarrow \R$ such that $\beta(t)=h(t)v$. The function
$h$ is locally Lipschitz since $\beta$ is, and injective since $\alpha$
is (on every cylindrical neighbourhood there is a natural time function).
Thus $h$ is strictly monotonous, and in fact strictly
increasing since otherwise $\beta$ would enter $J^-(0)$. Thus
$\beta'(t)=f(t)\beta(t)$ where $f(t):= \frac{h'(t)}{h(t)}\in
L^\infty_{\text{loc}}$. From here we may argue exactly as in 
\cite[Cor.\ 2.4.10]{Chrusciel_causality}: the function $r(s):=\int_0^s f(\tau)\,
d\tau$ is locally Lipschitz and strictly increasing, hence a
bijection from $[0,1]$ to some interval $[0,r_0]$.  Thus so is its
inverse $r\to s(r)$, and we obtain $\beta(s(r))' =
\beta'(s(r))/f(s(r)) = \beta(s(r))$ a.e., where the right hand side
is even continuous. It follows that in this parametrisation, $\beta$
is $C^1$ and in fact is a straight line in the null cone, hence
$\alpha$ can be parametrised as a null-geodesic segment, as claimed.
\end{proof}

%%%%%%%%%%%%%%%%%%%%%%%%%%%%%%%%%%%%%%%%%%%%%%%%%%%%%%%%%%%%%%%%%%%%%%%%%%%%%%%%%%%%%%%%%%%%%%%%%%%%%%%%%
%%%%%%%%%%%%%%%%%%%%%%%%%%%%%%%%%%%%%%%%%%%%%%%%%%%%%%%%%%%%%%%%%%%%%%%%%%%%%%%%%%%%%%%%%%%%%%%%%%%%%%%%%
\begin{Corollary} 
  The relation $\ll$ is open: if $p\ll q$ then there exist
  neighbourhoods $V$ of $p$ and $W$ of $q$ such that $p'\ll q'$ for all
  $p'\in V$ and $q'\in W$. In particular, for any $p\in M$, $I^+(p)$
  is open in $M$.
\end{Corollary}
%%%%%%%%%%%%%%%%%%%%%%%%%%%%%%%%%%%%%%%%%%%%%%%%%%%%%%%%%%%%%%%%%%%%%%%%%%%%%%%%%%%%%%%%%%%%%%%%%%%%%%%%%
%%%%%%%%%%%%%%%%%%%%%%%%%%%%%%%%%%%%%%%%%%%%%%%%%%%%%%%%%%%%%%%%%%%%%%%%%%%%%%%%%%%%%%%%%%%%%%%%%%%%%%%%%
\begin{proof} 
  Let $\alpha$ be a future-directed timelike curve from $p$ to $q$ and
  pick totally normal neighbourhoods $N_p$, $N_q$ of $p$, $q$ as in
  Th.\ \ref{lcb}. Now let $p'\in N_p$ and $q'\in N_q$ be points on
  $\alpha$. Then $V:=I^-(p',N_p)$ and $W:=I^+(q',N_q)$ have the
  required property.
\end{proof}
\medskip\noindent
From this we immediately conclude:
%%%%%%%%%%%%%%%%%%%%%%%%%%%%%%%%%%%%%%%%%%%%%%%%%%%%%%%%%%%%%%%%%%%%%%%%%%%%%%%%%%%%%%%%%%%%%%%%%%%%%%%%%
%%%%%%%%%%%%%%%%%%%%%%%%%%%%%%%%%%%%%%%%%%%%%%%%%%%%%%%%%%%%%%%%%%%%%%%%%%%%%%%%%%%%%%%%%%%%%%%%%%%%%%%%%
\begin{Corollary} Let $A\subseteq U \subseteq M$, where $U$ is open. Then
$$
I^+(A,U) = I^+(I^+(A,U)) = I^+(J^+(A,U)) = J^+(I^+(A,U)) \subseteq J^+(J^+(A,U)) = J^+(A,U)
$$
\end{Corollary}
\medskip\noindent
A consequence of Prop.\ \ref{nonnull} is that the causal future
of any $A\subseteq M$ consists (at most) of $A$, $I^+(A)$ and of
null-geodesics emanating from $A$:
%%%%%%%%%%%%%%%%%%%%%%%%%%%%%%%%%%%%%%%%%%%%%%%%%%%%%%%%%%%%%%%%%%%%%%%%%%%%%%%%%%%%%%%%%%%%%%%%%%%%%%%%%
%%%%%%%%%%%%%%%%%%%%%%%%%%%%%%%%%%%%%%%%%%%%%%%%%%%%%%%%%%%%%%%%%%%%%%%%%%%%%%%%%%%%%%%%%%%%%%%%%%%%%%%%%
\begin{Corollary}\label{on} 
  Let $A\subseteq M$ and let $\alpha$ be a causal curve from some
  $p\in A$ to some $q\in J^+(A)\setminus I^+(A)$. Then $\alpha$ is a
  null-geodesic that does not meet $I^+(A)$.
\end{Corollary}
%%%%%%%%%%%%%%%%%%%%%%%%%%%%%%%%%%%%%%%%%%%%%%%%%%%%%%%%%%%%%%%%%%%%%%%%%%%%%%%%%%%%%%%%%%%%%%%%%%%%%%%%%
%%%%%%%%%%%%%%%%%%%%%%%%%%%%%%%%%%%%%%%%%%%%%%%%%%%%%%%%%%%%%%%%%%%%%%%%%%%%%%%%%%%%%%%%%%%%%%%%%%%%%%%%%

\begin{proof} 
By Prop.\ \ref{nonnull}, $\alpha$ has to be a null curve. Moreover,
if $\alpha(t)\in I^+(A)$ for some $t$ then for some $a\in A$ we
would have $a\ll \alpha(t) \le q$, so $q\in I^+(A)$ by Prop.\ \ref{push-up1}, a
contradiction. Covering $\alpha$ by totally normal neighbourhoods as in Cor.\ \ref{lipisc1}
and applying Cor.\ \ref{boundary} gives the claim.
\end{proof}

\medskip\noindent
Following \cite[Lemma 14.2]{ON83} we next give a more refined
description of causality for totally normal sets. For this, recall
from the proof of \cite[Th.\ 4.1]{KSS} that the map $E: v\mapsto
(\pi(v),\exp_{\pi(v)}(v))$ is a homeomorphism from some open neighbourhood $S$
of the zero section in $TM$ onto an open neighbourhood $W$ of the
diagonal in $M\times M$.  If $U$ is totally normal as in Th.\ \ref{lcb} and such that $U\times
U \subseteq W$ then the map $U\times U \to TM$, $(p,q)\mapsto
\overrightarrow{pq}:=\exp_p^{-1}(q) = E^{-1}(p,q)$ is continuous.
%%%%%%%%%%%%%%%%%%%%%%%%%%%%%%%%%%%%%%%%%%%%%%%%%%%%%%%%%%%%%%%%%%%%%%%%%%%%%%%%%%%%%%%%%%%%%%%%%%%%%%%%%
%%%%%%%%%%%%%%%%%%%%%%%%%%%%%%%%%%%%%%%%%%%%%%%%%%%%%%%%%%%%%%%%%%%%%%%%%%%%%%%%%%%%%%%%%%%%%%%%%%%%%%%%%
\begin{Proposition} Let $U\subseteq M$ be totally normal as in Th.\ \ref{lcb}. 
\begin{itemize}
\item[(i)] Let $p$, $q\in U$. Then $q\in I^+(p,U)$ (resp.\ $\in J^+(p,U)$) if and only
if $\overrightarrow{pq}$ is future-directed timelike (resp.\ causal).
\item[(ii)] $J^+(p,U)$ is the closure of $I^+(p,U)$ relative to $U$.
\item[(iii)] The relation $\le$ is closed in $U\times U$.
\item[(iv)] If $K$ is a compact subset of $U$ and $\alpha: [0,b)\to K$ is causal, then $\alpha$
can be continuously extended to $[0,b]$. 
\end{itemize}
\end{Proposition} 
%%%%%%%%%%%%%%%%%%%%%%%%%%%%%%%%%%%%%%%%%%%%%%%%%%%%%%%%%%%%%%%%%%%%%%%%%%%%%%%%%%%%%%%%%%%%%%%%%%%%%%%%%
%%%%%%%%%%%%%%%%%%%%%%%%%%%%%%%%%%%%%%%%%%%%%%%%%%%%%%%%%%%%%%%%%%%%%%%%%%%%%%%%%%%%%%%%%%%%%%%%%%%%%%%%%
\begin{proof} 
  (i) and (ii) are immediate from Th.\ \ref{lcb}.

\medskip\noindent
  (iii) Let $p_n \le q_n$, $p_n\to p$, $q_n\to q$. By (i),
  $\overrightarrow{p_nq_n}$ is future-directed causal for all $n$. By
  continuity (\cite[Th.\ 4.1]{KSS}), therefore,
  $\langle\overrightarrow{pq},\overrightarrow{pq}\rangle \le 0$, so
  $\overrightarrow{pq}$ is future-directed causal as well.

\medskip\noindent
  (iv) Let $0< t_1 < t_2 < \dots \to b$. Since $K$ is compact,
  $\alpha(t_i)$ has an accumulation point $p$ and it remains to show
  that $p$ is the only accumulation point. Suppose that $q\not=p$ is
  also an accumulation point. Choose a subsequence $t_{i_k}$ such that
  $\alpha(t_{i_{2k}}) \to p$ and $\alpha(t_{i_{2k+1}}) \to q$. Then
  since $\alpha(t_{i_{2k}})\le \alpha(t_{i_{2k+1}}) \le
  \alpha(t_{i_{2k+2}})$, (iii) implies that $p\le q\le p$. By (i),
  then, $\overrightarrow{pq}$ would be both future- and past-directed,
  which is impossible.
\end{proof}
\medskip\noindent
From this, with the same proof as in \cite[Lemma 14.6]{ON83} we
obtain:
%%%%%%%%%%%%%%%%%%%%%%%%%%%%%%%%%%%%%%%%%%%%%%%%%%%%%%%%%%%%%%%%%%%%%%%%%%%%%%%%%%%%%%%%%%%%%%%%%%%%%%%%%
%%%%%%%%%%%%%%%%%%%%%%%%%%%%%%%%%%%%%%%%%%%%%%%%%%%%%%%%%%%%%%%%%%%%%%%%%%%%%%%%%%%%%%%%%%%%%%%%%%%%%%%%%
\begin{Corollary} Let $A\subseteq M$. Then
\begin{itemize}
\item[(i)] $J^+(A)^\circ=I^+(A)$.
\item[(ii)]  $J^+(A)\subseteq \overline{I^+(A)}$.
\item[(iii)] $J^+(A) = \overline{I^+(A)}$ if and only if $J^+(A)$ is closed.
\end{itemize}
\end{Corollary}
%%%%%%%%%%%%%%%%%%%%%%%%%%%%%%%%%%%%%%%%%%%%%%%%%%%%%%%%%%%%%%%%%%%%%%%%%%%%%%%%%%%%%%%%%%%%%%%%%%%%%%%%%
%%%%%%%%%%%%%%%%%%%%%%%%%%%%%%%%%%%%%%%%%%%%%%%%%%%%%%%%%%%%%%%%%%%%%%%%%%%%%%%%%%%%%%%%%%%%%%%%%%%%%%%%%
\medskip\noindent
Finally, as in the smooth case, one may introduce a notion of
causality also for general continuous curves (cf.\ \cite[p.\ 184]{HE},
\cite[Def.\ 8.2.1]{Kriele}):
%%%%%%%%%%%%%%%%%%%%%%%%%%%%%%%%%%%%%%%%%%%%%%%%%%%%%%%%%%%%%%%%%%%%%%%%%%%%%%%%%%%%%%%%%%%%%%%%%%%%%%%%%
%%%%%%%%%%%%%%%%%%%%%%%%%%%%%%%%%%%%%%%%%%%%%%%%%%%%%%%%%%%%%%%%%%%%%%%%%%%%%%%%%%%%%%%%%%%%%%%%%%%%%%%%%
\begin{definition}\label{contcaus} 
  A continuous curve $\alpha: I \to M$ is called future-directed
  causal (resp.\ timelike) if for every $t\in I$ there exists a
  totally normal neighbourhood $U$ of $\alpha(t)$ such that for any
  $s\in I$ with $\alpha(s)\in U$ and $s>t$, %(resp.$s<t$),
  $\alpha(s)\in J^+(\alpha(t))\setminus \{\alpha(t)\}$ (resp.\
  $\alpha(s)\in I^+(\alpha(t))\setminus \{\alpha(t)\}$), and analogously
  for $s<t$ with $J^-$ resp.\ $I^-$.
\end{definition}
%%%%%%%%%%%%%%%%%%%%%%%%%%%%%%%%%%%%%%%%%%%%%%%%%%%%%%%%%%%%%%%%%%%%%%%%%%%%%%%%%%%%%%%%%%%%%%%%%%%%%%%%%
%%%%%%%%%%%%%%%%%%%%%%%%%%%%%%%%%%%%%%%%%%%%%%%%%%%%%%%%%%%%%%%%%%%%%%%%%%%%%%%%%%%%%%%%%%%%%%%%%%%%%%%%%
\medskip\noindent
Then the proof of \cite[Lemma\ 8.2.1]{Kriele}) carries over to the
$C^{1,1}$-setting, showing that any continuous causal (resp.\
timelike) curve is locally Lipschitz.
\begin{remark}
While a continuous causal curve $\alpha$ need not
be a causal Lipschitz curve in the sense of our definition (cf.\ \cite[Rem.\ 1.28]{M}), 
it still follows that $\langle \alpha'(t),\alpha'(t)\rangle \le 0$ wherever $\alpha'(t)$ exists
(however, $\alpha'(t)$ might be $0$).

\medskip\noindent
To see this, consider first the case where $g$ is smooth. Set
$p:=\alpha(t)$, pick a normal neighbourhood $U$ around $p$ and set
$\beta:=\exp_p^{-1}\circ \alpha$. Then $\beta'(t)=\alpha'(t)$ and by
Def.\ \ref{contcaus} and Th.\ \ref{lcb}, $\beta(s)\in J^+(0)$
for $s>t$ small. Therefore, $\beta'(t)\in J^+(0)$, so $\langle
\alpha'(t),\alpha'(t)\rangle \le 0$. In the general case, where $g$ is
only supposed to be $C^{1,1}$, pick a regularisation $\hat g_\eps$ as
in Prop.\ \ref{CGapprox}. Then $\hat g_\eps(\alpha'(t),\alpha'(t)) \le
0$ for all $\eps$ by the above and letting $\eps\to 0$ gives the
claim.
\end{remark}

\section{Further aspects of causality theory}\label{further}
In the previous section we have shown that the fundamental
constructions of causality theory remain valid for
$C^{1,1}$-metrics. It was demonstrated by P.\ Chru\'sciel in
\cite{Chrusciel_causality} that to obtain a consistent causality
theory for $C^2$-metrics one needs two main ingredients: on the one
hand, a push-up Lemma, as given by Prop.\ \ref{push-up1},
\ref{nonnull}. The second pillar in the development of the theory is
the fact that accumulation curves of causal curves are causal again.
Here, if $\alpha_n:I\to M$ is a sequence of paths (parametrised
curves) then a path $\alpha:I\to M$ is called an accumulation curve of
the sequence $(\alpha_n)$ if there exists a subsequence
$(\alpha_{n_k})$ that converges to $\alpha$ uniformly on compact
subsets of $I$.  It was shown in \cite[Th.\ 1.6]{CG} that limit
stability of causal curves holds in fact even for continuous metrics:
\begin{Theorem}\label{accu} 
  Let $g$ be a $C^0$-Lorentzian metric on $M$, and let $\alpha_n: I\to
  M$ be a sequence of causal curves that accumulate at some $p\in M$
  ($\alpha_n(0)\to p$). Then there exists a causal curve $\alpha$ that
  is an accumulation curve of $\alpha_n$.
\end{Theorem}

%\textbf{Long version:}

\medskip\noindent
With these key tools at hand, and the results obtained so far,
causality theory for $C^{1,1}$-metrics can be further developed by
following the proofs given in \cite{Chrusciel_causality} for 
$C^2$-metrics. In the remainder of this section we list some 
main results that can be derived in this way.

Extendability of geodesics is characterised as follows (cf. 
\cite[Prop.\ 2.5.6]{Chrusciel_causality}):

\begin{Proposition}
Let $(M,g)$ be a spacetime with a $C^{1,1}$-Lorentzian metric $g$.
A geodesic $\alpha: I\to M$ is maximally extended as a geodesic if
and only if it is inextendible as a causal curve.
\end{Proposition}
%\begin{Theorem}
Furthermore, it is already shown in \cite[Th.\ 2.5.7]{Chrusciel_causality})
that even if the metric is merely supposed to be continuous, every
future directed causal (resp.\ timelike) curve possesses an 
inextendible causal (resp.\ timelike) extension of $\alpha$.
%\end{Theorem}
As a direct consequence of Cor.\ \ref{boundary} we obtain (cf.\ \cite[Prop.\ 2.6.9]{Chrusciel_causality}):
\begin{Proposition}
Let $g$ be a $C^{1,1}$-Lorentzian metric on $M$. If $\alpha$ is an achronal
causal curve, then $\alpha$ is a null geodesic.
\end{Proposition}
%Now looking at sequences of maximally extended geodesics or at sequences of
%achronal causal curves (curves such that $\forall t_1,t_2\in I,\ 
%\alpha(t_1)\notin I^+(\alpha(t_2))$), the following is easily obtained:
For sequences of curves, \cite[Prop.\ 2.6.8, Th.\ 2.6.10]{Chrusciel_causality}) give:
\begin{Proposition}
Let $g$ be a $C^{1,1}$-Lorentzian metric on $M$. If $\alpha_n: I\to M$ 
is a sequence of maximally extended geodesics accumulating at $\alpha$, 
then $\alpha$ is a maximally extended geodesic.
\end{Proposition}
\begin{Theorem}
Let $(M,g)$ be a spacetime with a $C^{1,1}$-Lorentzian metric $g$ and
let $\alpha_n:I\to M$ be a sequence of achronal causal curves accumulating 
at $\alpha$. Then $\alpha$ is achronal.
\end{Theorem}

\medskip\noindent
Causality conditions and notions such as domains of dependence and Cauchy
horizons can be defined independently of the regularity of
the metric. As an example of the interrelation of causality conditions
for metrics of low regularity, we mention
\cite[Prop.\ 2.7.4]{Chrusciel_causality}, which shows that if a spacetime with 
continuous metric is stably causal then it is strongly causal. 
Turning now to globally hyperbolic spacetimes, 
\cite[Prop.\ 2.8.1, Cor.\ 2.8.4, Th.\ 2.8.5]{Chrusciel_causality} give:
\begin{Proposition}
Let $(M,g)$ be a globally hyperbolic spacetime with $g$ a $C^{1,1}$-Lorentzian
metric and let $\alpha_n$ be a sequence of causal curves accumulating at 
both $p$ and $q$. Then there exists a causal curve $\alpha$ which is an
accumulation curve of the $\alpha_n$'s and passes through $p$ and $q$.
\end{Proposition}
\begin{Corollary}
If $M$ is a spacetime with a $C^{1,1}$-Lorentzian metric $g$ that is
globally hyperbolic, then 
\begin{equation*}
\overline{I^{\pm}(p)}=J^{\pm}(p). 
\end{equation*}
\end{Corollary}

\begin{Theorem}
For a globally hyperbolic spacetime $M$ with a $C^{1,1}$-metric $g$, if
$q\in I^+(p)$, resp. $q\in J^+(p)$, there exists a timelike, resp. causal,
future directed geodesic from $p$ to $q$.
\end{Theorem}
Moreover, the proof of \cite[Th.\ 2.9.9]{Chrusciel_causality} can be
adapted to show:
\begin{Theorem}
Let $M$ be a spacetime with a $C^{1,1}$-Lorentzian metric $g$ and let $S$
be an achronal hypersurface in $(M,g)$. Suppose that the interior 
$D^\circ_I(S)$ of the domain of dependence $D_I(S)$ is nonempty. Then 
$D^\circ_I(S)$ equipped with the metric obtained by restricting $g$ is
globally hyperbolic.
\end{Theorem}
Note that here the definition of domains of dependence is based on timelike
curves, as is the definition of Cauchy horizons. Finally, the analogue
of \cite[Prop.\ 2.10.6]{Chrusciel_causality} establishes the existence of
generators of Cauchy horizons :
\begin{Proposition}
Let $g$ be a $C^{1,1}$-Lorentzian metric on $M$ and let $S$ be a
spacelike $C^1$-hypersurface in $M$. For any point $p$
in the Cauchy horizon $H^+_I(S)$ there exists a past directed null geodesic
$\alpha_p\subset H^+_I(S)$ starting at $p$ which either does not have an 
endpoint in $M$, or has an endpoint in $\bar{S}\setminus S$.
\end{Proposition}

Based on these foundations, a deeper study of causality theory, in particular 
in the direction of singularity theorems for metrics of low regularity can
be undertaken. As detailed in \cite[Sec.\ 6.1]{Seno1}, this will require 
to solve a whole range of analytical problems that go beyond the results of this
paper, in particular concerning variational properties of curves, control of curvature
quantities, and a study of focal points. As already mentioned in the introduction,
we hope that the techniques developed in \cite{CleF,CG,KSS,M} as well as in this paper
can contribute to this task.

\medskip\noindent
{\bf Acknowledgements.}  We would like to thank James D.\ E. Grant for
helpful discussions.  The authors acknowledge the support of FWF
projects P23714 and P25326, as well as OeAD project WTZ CZ 15/2013.
We are indebted to the referees of this paper for several comments that
have led to substantial improvements in the presentation.

\end{document}